\documentclass[11pt,a4paper]{article}
\usepackage{amsmath}
\usepackage{amssymb}
\usepackage{amsthm}
\usepackage{amsfonts}
\usepackage{latexsym}
\usepackage{cite}
\usepackage{array}
\usepackage{mathtools}
\usepackage{enumerate}

\usepackage{url}

\usepackage[skip=10pt]{subcaption}

\usepackage[dvipsnames]{xcolor}
\usepackage{a4wide}

\usepackage{color}
\usepackage{float}

\usepackage[english]{babel}

\usepackage{mathtools}
\usepackage{todonotes}

\usepackage{graphicx}       

\usepackage[latin1]{inputenc}
\usepackage[T1]{fontenc}
\usepackage{microtype} 

\usepackage{cite}
\usepackage{array,colortbl}
\usepackage{amsmath,amsfonts,amssymb,bm,amsthm,mathtools,mathrsfs}
\DeclareFontFamily{U}{mathx}{\hyphenchar\font45}
\DeclareFontShape{U}{mathx}{m}{n}{
      <5> <6> <7> <8> <9> <10>
      <10.95> <12> <14.4> <17.28> <20.74> <24.88>
      mathx10
      }{}
\DeclareSymbolFont{mathx}{U}{mathx}{m}{n}
\DeclareFontSubstitution{U}{mathx}{m}{n}
\DeclareMathAccent{\widecheck}{0}{mathx}{"71}

\def\citep#1#2{\cite[{#1}]{#2}}

\usepackage{algorithm}
\usepackage{algorithmic}

\newcommand{\N}{{\mathbb{N}}} 
\newcommand{\R}{{\mathbb{R}}} 
\newcommand{\FF}{{\mathbb{F}}} 
\newcommand{\NN}{{\mathbb{N}}} 
\newcommand{\RR}{{\mathbb{R}}} 

\newcommand{\bsa}{{\boldsymbol{a}}}

\newcommand{\bsw}{{\boldsymbol{w}}}
\newcommand{\bsx}{{\boldsymbol{x}}}
\newcommand{\bsy}{{\boldsymbol{y}}}

\newcommand{\bszero}{{\boldsymbol{0}}} 

\newcommand{\bsgamma}{{\boldsymbol{\gamma}}}

\newcommand{\bsxi}{{\boldsymbol{\xi}}}

\newcommand{\calO}{{\mathcal{O}}}

\newcommand{\fraku}{{\mathfrak{u}}}

\newcommand{\setu}{{\mathfrak{u}}}
\newcommand{\setv}{{\mathfrak{v}}}

\newcommand{\abs}[1]{\left\vert#1\right\vert}

\newcommand{\floor}[1]{\left\lfloor #1 \right\rfloor} 
\newcommand{\rd}{\,\mathrm{d}} 

\usepackage{bookmark}
\makeatletter 
  \providecommand*{\toclevel@author}{999}
  \providecommand*{\toclevel@title}{0}
\makeatother

\usepackage{enumitem}

\theoremstyle{plain}
  \newtheorem{theorem}{Theorem}

  \newtheorem{corollary}{Corollary}
  
\theoremstyle{definition}
  \newtheorem{definition}{Definition}
  \newtheorem{example}{Example}
\theoremstyle{remark}
  \newtheorem{remark}{Remark}

\allowdisplaybreaks

\usepackage{enumitem}

\allowdisplaybreaks

\usepackage{todonotes}
\usepackage[noabbrev,capitalize]{cleveref}

\usepackage{tikz}
\usetikzlibrary{arrows.meta}
\usepackage{tikz-cd}

\newcommand{\No}{\mathbb{N}_0}
\newcommand{\Bf}[1]{\boldsymbol{#1}}
\newcommand{\xb}{\Bf {x}}

\newcommand{\cb}{\Bf {c}}
\newcommand{\kb}{\Bf {k}}
\newcommand{\ub}{\Bf {u}}

\newcommand{\CT}{\widetilde{C}}

\DeclareMathOperator{\rank}{rank}

\newcommand{\Fb}{\mathbb{F}_b}
\newcommand{\Disc}[1]

\def\cP{{\mathcal P}}

\title{Column reduced digital nets}

\author{V. Anupindi, P. Kritzer}

\begin{document}
\maketitle

\begin{abstract}
\noindent Digital nets provide an efficient way to generate integration nodes of quasi-Monte Carlo (QMC) rules. For certain applications, as e.g. in Uncertainty Quantification, we are interested in obtaining a speed-up in computing products of a matrix with the vectors corresponding to the nodes of a QMC rule. In the recent paper \textit{The fast reduced QMC matrix-vector product} (J. Comput. Appl. Math. 440, 115642, 2024), a speed up was obtained by using so-called reduced lattices and row reduced digital nets. 
In this work, we propose a different multiplication algorithm where we exploit the repetitive structure of column reduced digital nets instead of row reduced digital nets. This method has advantages over the previous one, as it facilitates the error analysis when using the integration nodes in a QMC rule.
We also provide an upper bound for the quality parameter of column reduced digital nets, and numerical tests to illustrate the efficiency of the new algorithm.
\end{abstract}

\section{Introduction}\label{sec:intro}

\subsection{The problem setting}
In many applications, such as in statistics, finance, and uncertainty quantification, we would like to numerically compute 
\begin{equation}\label{eq:integral}
 \int_D f(\bsx^\top A) \rd \mu(\bsx), 
\end{equation}
where $A$ is a real $s\times \tau$ matrix, by quasi-Monte Carlo (QMC) rules 
\begin{equation}\label{eq:QMC_approx}
Q_N (f):=\frac{1}{N}\, \sum_{k=0}^{N-1} f(\bsx_k^\top A), 
\end{equation}
where $\bsx_k =(x_k^{(1)},\ldots,x_k^{(s)})^\top$ are column vectors corresponding to the points used in the QMC rule. Problems of this kind particularly arise in some important applications in statistics and uncertainty quantification. For instance, this approach can be used when approximating the expected value of a function with a multivariate normal random variable with some given covariance matrix, or when approximating the expected value of the solution of a PDE with random coefficients, see, e.g., \cite{DKLS15}. 
In some cases the domain $D$ in \eqref{eq:integral} will be 
chosen as $D=[0,1]^s$. In this case, it is natural to directly use QMC sample points like lattice point sets (see \cite{DKP22}) or $(t,m,s)$-nets (see \cite{DP10}) as the points $\bsx_k$. This is the situation we shall mostly consider in the present paper, in particular with respect to the error analysis in Section \ref{subsec:err_ana}. 
However, in some applications, such as those mentioned above, the domain may be, e.g., $D=\RR^s$. Then the points $\bsx_k$ frequently are of the form $\bsx_k=\boldsymbol{\Phi}^{-1}(\bsy_k)$, $k\in\{0,1,\ldots,N-1\}$, where the $\bsy_k$ are the QMC sample points, and $\boldsymbol{\Phi}^{-1}$ is the inverse of the cumulative distribution function of a standard normal distribution, which is applied component-wise to vectors. In order to avoid that certain coordinates of the sample points are mapped to $\pm\infty$, one can first shift the $\bsy_k$ to the right by a sufficiently small quantity. Many results 
presented here also hold for this case, in particular the matrix product algorithm presented in Section \ref{subsec:red_mat_column}, see Remark \ref{remark:point_transform}. We also refer to \cite{DKLS15}, where a similar situation is studied using a different computational method.

Computing the vector-matrix products $\bsx_k^\top A$ for all $k\in\{0,\ldots,N-1\}$ takes $\mathcal{O}(N s\, \tau)$ operations. 
This problem is equivalent to computing the matrix-matrix product $XA$, where
\[
 X=\left[\bsx_0^\top,\bsx_1^\top,\ldots,\bsx_{N-1}^\top\right]^\top
\]
is the $N\times s$ matrix whose $k$-th row is $\bsx_k$. Computing $XA$
can be infeasible in situations where $s$ and $N$ are both large (which happens in many applications). 

In the paper \cite{DKLS15}, it is shown that when using particular types of QMC rules, 
the cost to evaluate $Q_N(f)$, as in \eqref{eq:QMC_approx}, can be reduced to only $\calO(\tau \,N \log N)$ operations provided that $\log N \ll s$.
This reduction in computational cost is achieved by a fast matrix-matrix
multiplication exploiting the fact that for specifically chosen point sets, 
such as (polynomial) lattice rules, 
the matrix $X$ can be re-ordered to be of circulant structure.
 
The recent paper \cite{DEHKL23} studies an alternative method to reduce the computation time by imposing a certain structure of the points $\bsx_0,\ldots,\bsx_{N-1}$. The key idea of this approach is to find situations in which the components of the points $\bsx_k$ have a certain repetitive structure, which then facilitates systematic fast computation of the products $\bsx_k^\top A$. This can be achieved by suitable modifications of (polynomial) lattice point sets using ideas from \cite{DKLP15}, but how to implement this idea for digital nets, which are more general than polynomial lattice point sets and among the most commonly used QMC node sets, 
is not straightforward. In \cite{DEHKL23}, the authors made a first attempt and studied a reduction of the computation time for digital $(t,m,s)$-nets by setting certain \textit{rows} of the generating matrices to zero (we refer to Section \ref{sec:digital} for 
the precise definition of digital nets and their generating matrices). The basic idea in \cite{DEHKL23} is that for each of the $s$ generating matrices $C_j^{(m)}$, $1\le j\le s$, of 
the digital net, we identify a so-called reduction index $w_j\in \mathbb{Z}$ and set the last $w_j$ rows of $C_j^{(m)}$ equal to zero. As shown in \cite{DEHKL23}, this introduces a certain repetitiveness in the entries of the matrix $X$ and speeds up the computation of the matrix-matrix product $XA$. We call such digital nets \textit{row reduced} digital nets. However, for assessing the quality of reduced nets when used in QMC rules, it is more natural to study the situation where certain \textit{columns} of the generating matrices are set to zero, since this directly corresponds to the reduced (polynomial) lattice point sets, resulting in the consideration of \textit{column reduced} digital nets. The idea of column reduced digital nets is to set the last $w_j$ columns of the generating matrix $C_j^{(m)}$, $1\le j\le s$, equal to zero, instead of setting rows equal to zero. 
Furthermore, in the present paper, we focus on digital nets that are obtained from \textit{digital sequences}, which implies additional structure in the generating matrices. Again, the approach of using column reduced digital nets yields a speed-up in the computation of $XA$, but as we will see below, it also makes it easier to assess the properties of the resulting column reduced digital nets than doing the same for row reduced digital nets. Furthermore, the error analysis for 
approximating \eqref{eq:integral} by \eqref{eq:QMC_approx} becomes easier. This idea was already mentioned (but not pursued) in \cite{DEHKL23}, and
this is what we intend to do in the present paper.

\subsection{Digital nets and sequences}\label{sec:digital}

In this section, we give the definitions of $(t,m,s)$-nets and $(t,s)$-sequences, the digital construction method for these, and shortly outline how to assess their quality. 

Let $\Fb$ be a finite field with $b$ elements, where $b$ is prime. We identify the elements of $\Fb$ with the set $\{0,1,\dots,b-1\}$. An \textit{elementary interval} in base $b$ and dimension $s$ is a half-open interval of the form $\prod_{j=1}^{s} [a_jb^{-d_j}, (a_j + 1)b^{-d_j} )$ where the $a_j,d_j$ are nonnegative integers with $0 \leq a_j < b^{d_j}$ for $1 \leq j \leq s$.

In the following, we recall the definition of $(t,m,s)$-nets and $(t,s)$-sequences, which have the property that the number of points in certain elementary intervals is proportional to their sizes. This guarantees a degree of uniform distribution of the point set in $[0,1)^s$, which is desirable when using such a point set in a QMC rule. 
For detailed discussions on $(t,m,s)$-nets and $(t,s)$-sequences, we refer to \cite{DP10, N92}.

\begin{definition}\label{def:tms_net}
For a given dimension $s \geq 1$ and nonnegative integers $t,m$ with $0 \leq t \leq m$, a \emph{$(t,m,s)$-net} in base $b$ is a point set $\mathcal{P} \subset [0,1)^s$ consisting of $b^m$ points such that any elementary interval in base $b$ with volume $b^{t-m}$ contains exactly $b^t$ points of $\mathcal{P}$.

A sequence $(\xb_{0},\xb_1,\dots)$ of points in $[0,1)^s$ is called a \emph{$(t,s)$-sequence} in base $b$ if for all integers $m\ge t$ and $k\geq 0$, the point set consisting of the points $\xb_{kb^m},\dots,\xb_{kb^m + b^m - 1}$ forms a $(t,m,s)$-net in base $b$. 
\end{definition}

Note that the lower the value of $t$ of a $(t,m,s)$-net or a $(t,s)$-sequence, the more uniformly the points are distributed in $[0,1)^s$, which is a desirable property when the point set is used as an integration node set in a QMC rule. This is the reason why $t$ is referred to as the \textit{quality parameter} of a net or sequence. 

A $(t,m,s)$-net is called \textit{strict}, if it does not fulfill the requirements of a $(t-1,m,s)$-net (for $t\ge 1$), and analogously for $(t,s)$-sequences. In general, any $(t,m,s)$-net is also a $(t+1,m,s)$-net for $t < m$.

We point out that it is, in general, a non-trivial combinatorial question of which values of $t$ can be reached for which configurations of the other parameters. We again refer to \cite{DP10, N92} for details.

A common way to generate $(t,m,s)$-nets and $(t,s)$-sequences is using the \emph{digital method}, which was first introduced by Niederreiter in \cite{N92a}. 
\begin{definition}
    A \emph{digital $(t,m,s)$-net over $\Fb$} is a $(t,m,s)$-net $\mathcal{P} = \{\xb_0,\dots,\xb_{b^{m}-1} \}$ where the points are constructed as follows. Let $C_1^{(m)},\dots,C_s^{(m)}$ in $\Fb^{m \times m}$ be matrices over $\Fb$. To generate the $k$-th point in $\mathcal{P}$, $0\le k\le b^m-1$, we use the $b$-adic expansion $k = \sum_{i=0}^{m-1} k_i b^i$ with digits $k_i \in \{0,\dots,b-1\}$ which we denote by $\overrightarrow{k} = (k_0,\dots,k_{m-1})^{\top}$. The $j$-th coordinate $x_{k,j}$ of $\xb_k = (x_{k,1},\dots,x_{k,s})$ is obtained by computing
    \[
    \overrightarrow{x}_{k,j} := C_j^{(m)} \,\overrightarrow{k},
    \]
    and then setting 
    \[
   x_{k,j}:=  \overrightarrow{x}_{k,j} \cdot (b^{-1},b^{-2},\ldots,b^{-m}).
    \]
    Similarly, a \emph{digital} $(t,s)$-sequence $\mathcal{S}$ over $\Fb$ 
    is generated by infinite matrices $C_1,\dots,C_s$, where 
    \begin{equation}\label{eq:def_Ci}
    C_j = (c_{i,r}^{(j)})_{i,r \in \N} \in \Fb^{\N \times \N}. 
    \end{equation}
    To generate the $k$-th point in $\mathcal{S}$, $k\ge 0$, we use the $b$-adic expansion $k = \sum_{i=0}^{\infty} k_i b^i$ with digits $k_i \in \{0,\dots,b-1\}$ which we denote by $\overrightarrow{k} = (k_0,k_1,\dots)^{\top}$. The 
    $j$-th coordinate $x_{k,j}$ of $\xb_k = (x_{k,1},\dots,x_{k,s})$ is obtained by computing
    \[
    \overrightarrow{x}_{k,j} :=  C_j^{(m)} \,\overrightarrow{k},
    \]
    and then setting 
    \[
   x_{k,j}:=  \overrightarrow{x}_{k,j} \cdot (b^{-1},b^{-2},\ldots).
    \]
\end{definition}

Note that from any digital $(t,s)$-sequence over $\Fb$ with generating matrices $C_1,\ldots,C_s$, we can, for $m\ge t$, derive a digital $(t,m,s)$-net over $\Fb$, simply by considering the point set generated by the left upper $m\times m$ submatrices $C_1^{(m)},\ldots,C_s^{(m)}$ of $C_1,\ldots,C_s$. This is equivalent to considering the first $b^m$ points of the $(t,s)$-sequence.

As pointed out above, the quality of a $(t,m,s)$-net or $(t,s)$-sequence is determined by its $t$-value. For digital $(t,m,s)$-nets and $(t,s)$-sequences, we can determine the $t$-value from rank conditions on the generating matrices, using a quantity that we shall refer to as the \textit{linear independence parameter}.

\begin{definition}
For any integers $1 \leq j \leq s$ and $m \geq 1$, let 
$C_{1}^{(m)}, C_2^{(m)},\ldots, C_s^{(m)}$ be $m \times m$ matrices over $\Fb$.
Then the \emph{linear independence parameter} $\rho_m 
(C_{1}^{(m)}, C_2^{(m)},\ldots, C_s^{(m)})$ is defined as the largest integer such that for any choice of $d_1,\dots,d_s \in \No$, with $d_1 +\dots+d_s = \rho_m$, we have that
\begin{align*}
    &\text{the first $d_1$ rows of $C_1^{(m)}$ together with}\\
    &\text{the first $d_2$ rows of $C_2^{(m)}$ together with}\\
    &\vdots \\
    &\text{the first $d_s$ rows of $C_s^{(m)}$}
\end{align*}
are linearly independent over $\Fb$. 
\end{definition}

It is known (see, e.g., \cite{N92,DP10}) that the generating matrices $C_{1}^{(m)}, C_2^{(m)},\ldots, C_s^{(m)}$ of a digital $(t,m,s)$-net over $\Fb$ satisfy 
\begin{equation} \label{eq:rho_t_relation}
    \rho_m (C_{1}^{(m)}, C_2^{(m)},\ldots, C_s^{(m)}) \ge m-t,
\end{equation}
where we have equality if the net is a strict $(t,m,s)$-net. Similarly, for the generating matrices $C_1,\ldots,C_s$ of a digital $(t,s)$-sequence over $\Fb$ we must have
$\rho_m(C_1^{(m)},\dots,C_s^{(m)}) \ge m-t$ for all $m\ge \max\{t,1\}$, where 
$C_j^{(m)}$ denotes the left upper $m\times m$ submatrix of $C_j$ for $j\in \{1,\ldots,s\}$.
Hence, for digital nets and sequences, their quality can be assessed by checking linear independence conditions on the rows of the generating matrices.

\section{The $t$-values  of column reduced digital nets} \label{sec:col_red_t}

\subsection{Column reduction for $(t,m,s)$-nets}

Now we turn towards the primary object of our study, which is the column reduced digital nets. We note that if we take a general digital $(t,m,s)$-net and set some columns of its generating matrices to zero, we cannot control the quality parameter of the reduced net.  However, since digital $(t,s)$-sequences require stronger conditions on their generating matrices, we can estimate the quality parameter of reduced digital $(t,m,s)$-nets derived from digital sequences by taking the nets generated by the left upper $m\times m$ submatrices of the generating matrices of the sequences.

For $m\ge t$, we consider the digital $(t,m,s)$-net generated by the matrices $C_1^{(m)},\dots,C_s^{(m)}$, derived via the above principle from a digital $(t,s)$-sequence with generating matrices $C_1,\ldots,C_s$, $C_j=(c_{i,r}^{(j)})$, $i,r\in\NN$.
 
Let $0=w_1 \le \cdots  \le w_s \in \No$, we call these numbers the \textit{reduction indices}, for the 
generating matrices $C_j^{(m)}$. We derive the corresponding reduced matrices $\CT_1^{(m)},\ldots,\CT_s^{(m)}$, with $\CT_j^{(m)}=(\widetilde{c}_{i,r}^{(j)})$, $i,r\in \{1,2,\ldots,m\}$, for $1\le j\le s$,
where
\begin{equation}\label{eq:def_Ci_red}
    \widetilde{c}_{i,r}^{(j)} = \begin{cases}
        c_{i,r}^{(j)} \quad &\text{ if } r \in \{ 1,\dots, m - \min{(m, w_j)}\},\\
        0 &\text{ if } r \in \{ m - \min{(m, w_j)} + 1,\dots,m \}.
    \end{cases}
\end{equation}

That is, the first $m - \min{(m,w_j)}$ columns of $\CT_j^{(m)}$ are the same as the columns of the matrix $C_j^{(m)}$, and we set the last $\min{(m,w_j)}$ columns to zero, i.e, if $w_j < m $,
\[ 
\CT_j^{(m)} = \begin{pmatrix}
c_{1,1}^{(j)}      & \dots &  c_{1, (m-w_{j})}^{(j)} & 0 & \dots &0  & \\
\vdots & \ddots &       &     \vdots  &  \vdots  &  \\
c_{(m-w_{j}),1}^{(j)}     &   \dots     & c_{(m-w_{j}),(m-w_{j})}^{(j)}     &  0 &\dots    & 0 &  \\
\vdots & \ddots &  \vdots     &   \vdots  &  & \vdots \\ 
c_{m,1}^{(j)}     &   \dots     & c_{m,(m-w_{j})}^{(j)}     &  0 &\dots    & 0 &  \\
\end{pmatrix}.
\]

We are interested in estimating the quality parameter of the digital net generated by the $\CT_j^{(m)}$. 

Apart from the main motivation outlined in Section \ref{sec:intro}, there is another computational advantage of using 
column reduced digital nets. Indeed, by the general construction principle of digital point sets, the generating matrices of a digital net or sequence are multiplied over $\Fb$ by vectors representing the digits of the indices of the elements of the point set. By replacing the matrices $C_j^{(m)}$ by $\CT_j^{(m)}$, we increase the sparsity of the generating matrices, which saves computation time in the generation of the point set.

\begin{theorem}\label{lem:t_value_red}
Let $\cP$ be a digital $(t,m,s)$-net over $\Fb$ with generating matrices 
$C_1^{(m)}, \dots ,C_s^{(m)}$ derived from a digital $(t,s)$-sequence over $\Fb$, 
where we assume that $m\ge t$. 
Let $\CT_1^{(m)},\ldots,\CT_s^{(m)}$ be as defined in \eqref{eq:def_Ci_red} with respect to reduction indices $0= w_1 \leq \dots \leq w_s$ and let $\widetilde{t}$ 
be the minimal quality parameter of the net generated by the $\CT_j^{(m)}$. Then,
\begin{equation}\label{eq:claim_rho_CT}
    \max\{0,m-w_s-t\} \leq \rho_m \left(\CT_1^{(m)},\dots,\CT_s^{(m)}\right) \leq \max\{0,m - w_s\},
\end{equation}
and $ \widetilde{t} \le \min\{m,w_s +t\}$.

\medskip

Furthermore, if $\cP$ is a strict digital $(t,m,s)$-net, it is true that
\begin{equation}\label{eq:claim_rho_CT_strict}
    \rho_m \left(\CT_1^{(m)},\dots,\CT_s^{(m)}\right) \leq \max\{0,m - \max\{t,w_s\}\}.
\end{equation}
\end{theorem}
\begin{proof}
We note that we have $m\ge t$ by assumption. If $w_s\ge m$, then we trivially have 
$\rho_m \left(\CT_1^{(m)},\dots,\CT_s^{(m)}\right)=0$, as $\CT_s^{(m)}$ only contains zeros, and \eqref{eq:claim_rho_CT} holds.

Therefore, we will assume for the rest of the proof that $w_s < m$. 

We prove the second inequality in \eqref{eq:claim_rho_CT} first.
We have 
\begin{equation*} 
\CT_s^{(m)} = 
\begin{pmatrix}
  c_{1,1}^{(s)}      & \dots &  c_{1,(m-w_{s})}^{(s)} & 0 & \dots & 0 & \\
  \vdots & \ddots &       &     \vdots  &   & \vdots&  \\
  c_{(m-w_{s}),1}^{(s)}     &  \dots      & c_{(m-w_{s}),(m-w_{s})}^{(s)}     &  0     &\dots  & 0 & \\
  \vdots & \ddots  &  \vdots     &   \vdots   & & \vdots \\
  c_{m,1}^{(s)}      & \dots &  c_{m, (m-w_{s})}^{(s)} & 0 & \dots & 0 & \\  
\end{pmatrix}.
\end{equation*} 
Let $D$ be the matrix containing the first $d_1$ rows of $\CT_1^{(m)}$, the first $d_2$ rows of $\CT_2^{(m)}$, etc., up to the first $d_s$ rows of $\CT_s^{(m)}$, 
where $d_1,\ldots,d_s$ are nonnegative integers satisfying $d_1+\cdots + d_s= m-w_s$. For 
the special choice $(d_1,\dots,d_s) = (0,\dots,0,m-w_s)$, we have $\rank(D) = \rank(\CT_s^{(m)}) \leq m-w_s$. Therefore,
\begin{equation*}
   \rho_m\left(\CT_1^{(m)},\dots,\CT_s^{(m)}\right) \leq m-w_s.
\end{equation*}

Now we prove the first inequality in \eqref{eq:claim_rho_CT}. If $m-w_s-t<0$, the inequality is trivial. 

Otherwise, i.e., if $m-w_s\ge t$, we know that 
\begin{equation}\label{eq:rho_C}
    \rho_k \left(C_1^{(k)},\dots,C_s^{(k)}\right) \geq k-t,
\end{equation}
for any $k\ge t$, since our net is derived from a digital $(t,s)$-sequence. Here, $C_j^{(k)}$, $1\le j\le s$, denotes the left upper $k\times k$ submatrix of $C_j$. In particular, we observe that for the left upper $(m-w_s) \times (m-w_s)$ submatrices  
of $C_1,\ldots,C_s$,
$$
\rho_{(m-w_s)}(C_1^{(m-w_s)},\dots,C_s^{(m-w_s)}) \geq m-w_s - t.
$$

We now consider arbitrary integers $d_1,\dots,d_s \geq 0$ with $d_1 + \dots + d_s =  m- w_s-t$. 
Let $\kb_{i}^{(j)}$ denote the $i$-th row vector of $C_j^{(m-w_s)} \in \Fb^{(m-w_s) \times (m-w_s)}$.
We know that 
\begin{equation}\label{eq:k_vec_indepen}
    \kb_{1}^{(1)},\dots,\kb_{d_1}^{(1)}, \kb_{1}^{(2)},\dots,\kb_{d_2}^{(2)},\dots, \dots,  \kb_{1}^{(s)},\dots,\kb_{d_s}^{(s)}
\end{equation}
are linearly independent over $\Fb$.
Let $\cb_{i}^{(j)}$ denote the $i$-th row vector of $\CT_i^{(m)} \in \Fb^{m \times m}$. We observe that for $1\leq i \leq m-w_s$,
\begin{equation*}
    \cb_{i}^{(j)} = (\kb_{i}^{(j)},\ub_{i}^{(j)}) \in \Fb^{1 \times m },
\end{equation*}
where the $\kb_{i}^{(j)} $ are as above and $\ub_{i}^{(j)} \in \Fb^{1 \times w_s}$.
 
The row vectors
\begin{equation}\label{eq:c_vec_indepen}
    \cb_{1}^{(1)},\dots,\cb_{d_1}^{(1)}, \cb_{1}^{(2)},\dots,\cb_{d_2}^{(2)}, \dots,  \dots,\cb_{1}^{(s)},\dots,\cb_{d_s}^{(s)}
\end{equation}
are linearly independent, since otherwise the row vectors in \eqref{eq:k_vec_indepen}, which are projections of $\cb_i^{(j)}$ onto the first $m-w_s$ entries, would be linearly dependent. 
Therefore,
\begin{equation*}
    \rho_m\left(\CT_1^{(m)},\dots,\CT_s^{(m)}\right) \geq m-w_s-t.
\end{equation*}
This concludes the proof of \eqref{eq:claim_rho_CT}. Using \eqref{eq:rho_t_relation} and the lower bound in \eqref{eq:claim_rho_CT}, we obtain the upper bound for $\widetilde{t}$.

\medskip

It remains to show \eqref{eq:claim_rho_CT_strict}.

Let $D$ be the matrix containing the first $d_1$ rows of $\CT_1^{(m)}$, the first $d_2$ rows of $\CT_2^{(m)}$, etc., up to the first $d_s$ rows of $\CT_s^{(m)}$, 
where $d_1,\ldots,d_s$ are nonnegative integers. As above, for 
the special choice $(d_1,\dots,d_s) = (0,\dots,0,m-w_s)$, we have 
$\rank(D) = \rank(\CT_s^{(m)}) \leq (m-w_s)$. So,
\begin{equation*}
   \rho_m \left(\CT_1^{(m)},\dots,\CT_s^{(m)}\right) \leq m-w_s.
\end{equation*}
However, since we assume that $\cP$ is a strict digital $(t,m,s)$-net in this part of the proof, there must 
exist a choice of $(d_1,\dots,d_s)$ with $d_1 + \cdots + d_s=m-t+1$ such that the corresponding rows of 
$C_1^{(m)},\ldots,C_s^{(m)}$ are linearly dependent, and therefore also the corresponding rows of $\CT_1^{(m)},\ldots,\CT_s^{(m)}$
are linearly dependent. This yields
\begin{equation*}
  \rho_m  \left(\CT_1^{(m)},\dots,\CT_s^{(m)}\right)  \leq m-t,
\end{equation*}
so we must have
\begin{equation*}
  \rho_m \left(\CT_1^{(m)},\dots,\CT_s^{(m)}\right)  \leq m-\max\{t,w_s\}.
\end{equation*}
\end{proof}

\begin{remark}
    For $t=0$ and $w_s < m$ in \cref{lem:t_value_red}, we obtain equality in \eqref{eq:claim_rho_CT} and therefore 
\begin{equation*}
        \rho_m \left(\CT_1^{(m)},\dots,\CT_s^{(m)}\right) = m-w_s,
    \end{equation*}
    and $\widetilde{t} = w_s$.
\end{remark}

\begin{remark}\label{rem:lower_sharp}
 We now give an example that illustrates that the lower bound in Theorem \ref{lem:t_value_red} is sharp. 
 
 Assume that $Q$ is a digital $(0,2)$-sequence with generating matrices $D_1$ and $D_2$ (examples of $Q$ exist, 
 e.g., by choosing as $Q$ a Niederreiter sequence, see \cite{N92a}). 
 
 From $Q$, we construct a digital $(t,2)$-sequence $P$, by prepending exactly $t$ zero columns to both $D_1$ and $D_2$. That is, 
 we construct new generating matrices $C_j$, $j\in \{1,2\}$, such that
 \[
  C_j:=\begin{pmatrix}
        0 & \ldots & 0  &\vline &  & & \\
        \vdots & \vdots & \vdots  &\vline& & D_j & \\
        0 & \ldots & 0 &\vline &  &  & \\
        \vdots & \vdots & \vdots &\vline & & & \\
       \end{pmatrix}.
 \]
 It is easily checked that $C_1, C_2$ generate a digital $(t,2)$-sequence; indeed, let $m\ge t$ be arbitrarily chosen but fixed. 
 Then the matrices $C_1^{(m)}, C_2^{(m)}$ contain the matrices $D_1^{(m-t)}, D_2^{(m-t)}$ as submatrices. As $D_1, D_2$ generate a $(0,2)$-sequence,
 for any $d_1,d_2 \in \NN_0$ with $d_1 + d_2 =m-t$ the first $d_1$ rows of $D_1^{(m-t)}$ together with the first $d_2$ rows of $D_2^{(m-t)}$ must 
 be linearly independent, so also the corresponding rows of $C_1^{(m)}$ and $C_2^{(m)}$ (with zeros prepended) must be linearly independent. This 
 establishes that $C_1$ and $C_2$ generate a $(t,2)$-sequence. 
 
 Let now $m\ge t$, and let $w_1=0$, and $w_2\ge w_1$ be reduction indices such that $m-w_2-t\ge 0$. Then $\CT_1^{(m)}=C_1^{(m)}$, and
 \[
  \CT_2^{(m)}=\begin{pmatrix}
        0 & \ldots & 0  &\vline &  & &   \vline & 0 & \ldots  & 0\\
        \vdots & \vdots & \vdots  &\vline& & D_2^{(m\times (m-t-w_2))} & \vline & \vdots & \vdots & \vdots\\
        0 & \ldots & 0 &\vline &  &  & \vline & 0 & \ldots & 0\\
       \end{pmatrix},
 \]
where $D_2^{(m\times (m-t-w_2))}$ denotes the left upper $m\times (m-t-w_2)$ submatrix of $D_2$. By Theorem \ref{lem:t_value_red}, we know that 
$\rho_m \left(\CT_1^{(m)},\CT_2^{(m)}\right) \ge m-t-w_2$. However, $\rho_m \left(\CT_1^{(m)},\CT_2^{(m)}\right) > m-t-w_2$ cannot hold since the first $m-t-w_2+1$ rows of $\CT_2^{(m)}$ must be linearly 
dependent. 

This implies that the lower bound in Theorem \ref{lem:t_value_red} is sharp. 
\end{remark}

\begin{remark}\label{rem:upper_sharp}

Next, we provide an example showing that the upper bound \eqref{eq:claim_rho_CT_strict} for strict digital nets in Theorem \ref{lem:t_value_red} is sharp. 

We use the same notation as in Remark \ref{rem:lower_sharp}.
We again start with the digital $(0,2)$-sequence $Q$. Again, we transform $Q$ into a $(t,s)$-sequence, now called $R$, with generating matrices 
$E_1$ and $E_2$. For $E_1$, we take the first generating matrix of $P$ from above, i.e., $E_1=C_1$. Furthermore, we choose $E_2$ as
\[
 E_2:=\begin{pmatrix}
    & & &\vline & 0 & \ldots & 0 & \ldots \\
    & D_2^{(t)} & & \vline & \vdots & \vdots & \vdots & \vdots \\
    & & &\vline & 0 & \ldots & 0 & \ldots \\
    \hline
    0 & \ldots & 0  &\vline &  & & &\\
        \vdots & \vdots & \vdots  &\vline& & D_2 & &\\
        0 & \ldots & 0 &\vline &  &  & &\\
        \vdots & \vdots & \vdots &\vline & & & &\\
 \end{pmatrix},
\]
where $D_2^{(t)}$ is the left upper $t\times t$ submatrix of $D_2$.
First, note that $R$ really is a strict $(t,2)$-sequence. Indeed, if we consider the matrix $E_1^{(m)}$ for $m< t$, this matrix only contains zeros, 
so the quality parameter of $R$ must be at least $t$. On the other hand, let $m\ge t$ and consider the matrices $E_1^{(m)}$ and 
$E_2^{(m)}$. Choose $d_1,d_2\ge 0$ such that $d_1 + d_2 =m-t$, and consider the first $d_1$ rows of $E_1^{(m)}$ together with the 
first $d_2$ rows of $E_2^{(m)}$. We distinguish two cases.
\begin{itemize}
 \item If $d_2 \le t$, then it is obvious that the first $d_1$ rows of $E_1^{(m)}$ together with the 
first $d_2$ rows of $E_2^{(m)}$ are linearly independent, as $D_1$ and $D_2$ generate a $(0,2)$-sequence.

\item If $d_2 >t$, we proceed as follows. Assume to the contrary that the first  $d_1$ rows of $E_1^{(m)}$ together with the 
first $d_2$ rows of $E_2^{(m)}$ were not linearly independent. By the structure of $E_1$ and $E_2$, this would immediately imply that 
the first $d_1$ rows of $D_1^{(m-t)}$ together with the first $d_2-t$ rows of $D_2^{(m-t)}$ are not linearly independent, where 
$d_1 + d_2-t = m-2t$, which would be a contradiction 
to the property that $D_1$ and $D_2$ generate a digital $(0,2)$-sequence. 
\end{itemize}

 Let now again $m\ge t$, and let $w_1=0$, and $w_2\ge w_1$ be reduction indices such that $m-w_2-t\ge 0$. Then $\widetilde{E}_1^{(m)}:= E_1^{(m)}$, and
 \[
  \widetilde{E}_2^{(m)}:=\begin{pmatrix}
    & & &\vline & 0 & \ldots & 0  &\vline & 0 & \ldots & 0\\
    & D_2^{(t)} & & \vline & \vdots & \vdots & \vdots &\vline &  \vdots & \vdots & \vdots\\
    & & &\vline & 0 & \ldots & 0 &\vline &  0 & \ldots & 0\\
    \hline
    0 & \ldots & 0  &\vline &  & & &\vline &  0 & \ldots & 0\\
        \vdots & \vdots & \vdots  &\vline& & D_2^{((m-t)\times (m-t-w_2))} & &\vline &  \vdots & \vdots & \vdots\\
        0 & \ldots & 0 &\vline &  &  & &\vline &  0 & \ldots & 0\\
 \end{pmatrix}.
 \]

We again distinguish two cases.
 
\paragraph{\textbf{Case 1:} \boldmath{$\max\{t,w_2\}=w_2$}.} 
We claim that $\rho_m \left(\widetilde{E}_1^{(m)},\widetilde{E}_2^{(m)}\right)=m-w_2$. To this end, let $d_1,d_2\ge 0$ such that $d_1+d_2= m-w_2$, which implies 
that $d_1$ and $d_2$ are both not larger than $m-t$. 
Then, we consider two sub-cases.
\begin{itemize}
 \item If $d_2 \le t$, it is clear because of the structure of the matrices that the first $d_1$ rows of $\widetilde{E}_1^{(m)}$ together with the 
first $d_2$ rows of $\widetilde{E}_2^{(m)}$ are linearly independent, as $D_1$ and $D_2$ generate a $(0,2)$-sequence. This is guaranteed 
since we know that $d_1$ and $d_2$ are both not larger than $m-t$. 

\item If $d_2 >t$, we proceed as follows. Assume to the contrary that the first  $d_1$ rows of $\widetilde{E}_1^{(m)}$ together with the 
first $d_2$ rows of $\widetilde{E}_2^{(m)}$ were not linearly independent. 

By the structure of $\widetilde{E}_1^{(m)}$ and $\widetilde{E}_2^{(m)}$, this would immediately imply that 
the first $d_1$ rows of $D_1^{(m-t)}$ together with the first $d_2-t$ rows of $D_2^{((m-t)\times (m-t-w_2))}$ are not linearly independent, where 
$d_1 + d_2-t = m-t-w_2$. Note, however, that $D_1^{(m-t)}$ contains $D_1^{(m-t-w_2)}$ as its left upper submatrix, and also $D_2^{((m-t)\times (m-t-w_2))}$ contains
$D_2^{(m-t-w_2)}$ as its left upper submatrix. By the property that $D_1$ and $D_2$ generate a $(0,2)$-sequence, and by the assumption that $m-w_2\ge t$, 
the first $d_1$ rows of $D_1^{(m-t-w_2)}$ together with the first $d_2-t$ rows of $D_2^{(m-t-w_2)}$ must be linearly independent. The same must, however, 
then also hold for the corresponding rows of $D_1^{(m-t)}$ and $D_2^{((m-t)\times (m-t-w_2))}$, which yields a contradiction. 
\end{itemize}
Hence we have shown that $\rho_m \left(\widetilde{E}_1^{(m)},\widetilde{E}_2^{(m)}\right) \ge m-w_2$, and by Theorem \ref{lem:t_value_red} we must actually have $\rho_m \left(\widetilde{E}_1^{(m)},\widetilde{E}_2^{(m)}\right)= m-w_2$.

\paragraph{\textbf{Case 2:} \boldmath{$\max\{t,w_2\}=t$}.} 
We claim that $\rho_m \left(\widetilde{E}_1^{(m)},\widetilde{E}_2^{(m)}\right)=m-t$. To this end, let $d_1,d_2\ge 0$ such that $d_1+d_2= m-t$. Also here, we distinguish 
two sub-cases. 
\begin{itemize}
 \item If $d_2 \le t$, it is obvious that the first $d_1$ rows of $\widetilde{E}_1^{(m)}$ together with the 
first $d_2$ rows of $\widetilde{E}_2^{(m)}$ are linearly independent, as $D_1$ and $D_2$ generate a $(0,2)$-sequence. This is guaranteed 
since we know that $d_1$ and $d_2$ are both not larger than $m-t$. 

\item If $d_2 >t$, we proceed as follows. Assume to the contrary that the first  $d_1$ rows of $\widetilde{E}_1^{(m)}$ together with the 
first $d_2$ rows of $\widetilde{E}_2^{(m)}$ were not linearly independent. 

By the structure of $\widetilde{E}_1^{(m)}$ and $\widetilde{E}_2^{(m)}$, this would immediately imply that 
the first $d_1$ rows of $D_1^{(m-t)}$ together with the first $d_2-t$ rows of $D_2^{((m-t)\times (m-t-w_2))}$ are not linearly independent, where 
$d_1 + d_2-t = m-2t\le m-t-w_2$. Note, however, that $D_1^{(m-t)}$ contains $D_1^{(m-t-w_2)}$ as its left upper submatrix, and also $D_2^{((m-t)\times (m-t-w_2))}$ contains
$D_2^{(m-t-w_2)}$ as its left upper submatrix. By the property that $D_1$ and $D_2$ generate a $(0,2)$-sequence, by the 
fact that $d_1 +d_2 \le m-t-w_2$, and by the assumption that $m-w_2\ge t$, 
the first $d_1$ rows of $D_1^{(m-t-w_2)}$ together with the first $d_2$ rows of $D_2^{(m-t-w_2)}$ must be linearly independent. The same must, however, 
then also hold for the corresponding rows of $D_1^{(m-t)}$ and $D_2^{((m-t)\times (m-t-w_2))}$, which yields a contradiction. 
\end{itemize}
In summary, we have shown that \eqref{eq:claim_rho_CT_strict} is sharp for strict digital nets.
\end{remark}

\subsection{Column reduction for $(t,m,{\bf{\textit{e}}},s)$-nets}

In \cite{T13} Tezuka introduced the concept of $(t,m,\Bf{e},s)$-nets, which are a generalization of $(t,m,s)$-nets. In this section, we briefly look at the quality parameter of column reduced nets under this generalized definition of nets. However, for the rest of the paper, we shall then stick to the notion of $(t,m,s)$-nets again.

\begin{definition}\label{def:tmes_nets}
    Let $\Bf{e} = (e_1, \dots, e_s)$ and $\Bf{d} = (d_1, \dots, d_s)$ be integer vectors with $e_i \geq 1$ and $d_i \geq 0$ for $i \in\{1,\ldots,s\}$, where $s \geq 1$ is the dimension. Let $t,m$ be non-negative integers with $0 \leq t \leq m$. A point set $\cP \subset [0,1)^s$ with $b^m$ points is called a $(t,m,\Bf{e},s)$-net in base $b$ if every elementary interval $J \subseteq [0,1)^s$ of volume $b^{t-m}$ and of the form
    \[
    J = \prod_{j=1}^{s} \Big[ \frac{a_j}{b^{e_jd_j}}, \frac{a_j + 1}{b^{e_jd_j}} \Big)
    \]
    contains exactly $b^t$ points of $\cP$, where $0 \leq a_j < b^{e_jd_j}$ for $j\in\{1,\ldots,s\}$ and $\Bf{d}$ satisfies the equation $e_1d_1 + \dots + e_sd_s = m-t$. 
\end{definition}
If we choose $\Bf{e} = (1,\dots,1) \in \NN^s$, we obtain the classical definition of a $(t,m,s)$-net as given in \cref{def:tms_net}. For $(t,m,s)$-nets, we have the propagation rule that a $(t,m,s)$-net in base $b$ is also a $(v,m,s)$-net in base $b$ for any integer $v$ with $t \leq v \leq m$. However, with the above definition of $(t,m,\Bf{e}, s)$-nets, we do not have this propagation rule. In \cite{HN13}, the authors provided a revised definition of $(t,m,\Bf{e},s)$-nets which ensures the above mentioned propagation rule. In this section, however, we work with the original definition provided by Tezuka in \cite{T13}.

We note that all $(t,m,s)$-nets are also $(t,m, \Bf{e},s)$-nets, however for certain values of $\Bf{e}$, we can obtain a lower $t$ value for the corresponding $(t,m, \Bf{e},s)$-net. In particular, for column-reduced digital nets, we can find certain examples where at least for some choices of $\Bf{e}$, the reduced net retains the original quality parameter $t$. Let us give some examples.
\begin{example}
    Let $b=2, s=2, m=4$ and consider the $(0,4,2)$-net derived from the Sobol' sequence, given by the generating matrices
\begin{equation*}
C_1 = \begin{pmatrix}
        1 & 0 & 0 & 0 \\
        0 & 1 & 0 & 0 \\
        0 & 0 & 1 & 0 \\
        0 & 0 & 0 & 1
       \end{pmatrix}, \quad
C_2 = \begin{pmatrix}
        1 & 1 & 1 & 1 \\
        0 & 1 & 0 & 1 \\
        0 & 0 & 1 & 1 \\
        0 & 0 & 0 & 1
       \end{pmatrix} .
\end{equation*}
Let $w_1 = 0$ and $w_2 = 1$, then the resulting column reduced digital net is a $(1,4,2)$-net according to \cref{lem:t_value_red}. However, for $\Bf{e} = (e_1,e_2)$ chosen such that $(e_1d_1, e_2d_2) = (4,0)$ is the only solution to the equation $e_1d_1 + e_2d_2 = 4$, we obtain a column reduced net that is still a $(0,4,\Bf{e},2)$-net. This is because the only elementary intervals $J$ that satisfy the conditions in \cref{def:tmes_nets} are of the form
\[
J = [a_1/2^4, (a_1 + 1)/2^4) \times [0,1).
\]
Thus, the net property depends only on the first coordinates in the point set $\cP$, and since we do not set any columns of $C_1$ to zero, i.e., $w_1 = 0$, the resulting column reduced net is a $(0,4,\Bf{e},2)$-net. Some concrete choices for $\Bf{e}$ with the above property are $(e_1,e_2) = (2,3)$, or $(e_1,e_2) = (1,e_2)$ where $e_2 > 4$, or $(e_1,e_2) = (4,e_2)$ where $e_2 = 3 $ or $ e_2 > 4$. 
\end{example}
In general, given a $(t,m,s)$-net $\cP$ derived from a $(t,s)$-sequence and reduction indices $0=w_1 \leq w_2 \leq \dots \leq w_s$, for $\Bf{e} = (e_1,\dots,e_s) = (m-t,k,\dots,k)$, where either $1< k < m-t$ with $\gcd(k,m-t) = 1$ or $k > m-t >0$, the only solution to the equation $\sum_{j=1}^{s} e_jd_j = m-t$ is $\Bf{d} = (1,0,\dots,0)$. Thus, the column reduced net $\widetilde{\cP}$ is an $(m-t,m,\Bf{e},s)$-net for the above choice of $\Bf{e}$.

One might also consider digital $(t,m,s)$-nets generated by $C_1,\dots,C_s$ where the $C_j$ for $2\leq j \leq s$ are derived from a $(t,s)$-sequence but $C_1$ is not necessarily derived from a digital sequence, since we usually choose the first reduction index $w_1 = 0$. In this case, one could perhaps find more choices of $\Bf{e} = (e_1,\dots,e_s)$ such that the corresponding column reduced digital net is a $(t,m,\Bf{e},s)$-net, depending on the reduction indices $w_2,\dots,w_s$. 

A general and complete analysis of reduced $(t,m,\Bf{e},s)$-nets could be an interesting subject for future research.

\section{Projections of column reduced digital nets} \label{sec:proj_col_red}

Due to the important role of the $t$-value, one sometimes also considers a 
slightly refined notion of a $(t,m,s)$-net, which is then referred to as a 
$((t_\setu)_{\setu \subseteq [s]},m,s)$-net, where $[s]:=\{1,\ldots,s\}$.
The latter notion means that for any $\setu\neq\emptyset$, $\setu\subseteq [s]$, the projection of the net onto those components with indices in $\setu$
is a $(t_\setu,m,\abs{\setu})$-net. The notion of a $((t_\setu)_{\setu \subseteq [s]},s)$-sequence is defined analogously. Moreover, for $\setu\neq \emptyset$, we write 
$\overline{\setu}:=\max (\setu)$.

If we assume (which we always do in this paper) that the reduction indices 
satisfy $0=w_1\le w_2 \le \cdots \le w_s$, then, for any non-empty $\setu\subseteq [s]$, the reduction index $w_{\overline{\setu}}$ is the largest among all reduction indices corresponding to $\setu$. This yields the 
following adaption of Theorem \ref{lem:t_value_red}, which obviously can be shown in the same manner.

\begin{corollary}\label{lem:t_value_red_proj}
Let $\cP$ be a digital $((t_\setu)_{\setu \subseteq [s]},m,s)$-net over $\Fb$ with generating matrices $C_1^{(m)}, \dots ,C_s^{(m)}$, which has been derived from a digital $((t_\setu)_{\setu \subseteq [s]},s)$-sequence, 
where we assume that $m\ge t$. 
Let $\CT_1^{(m)},\ldots,\CT_s^{(m)}$ be the reduced generating matrices with respect to reduction indices $0= w_1 \leq \dots \leq w_s$ and let $(\widetilde{t}_\setu)_{\setu \subseteq [s]}$ 
be the minimal quality parameters of the projections of the net generated by the $\CT_j^{(m)}$. Then, for every non-empty $\setu\subseteq [s]$,
\begin{equation*}
    \max\{0,m-w_{\overline{\setu}}-t_{\setu}\} \leq \rho_m((\CT_j^{(m)})_{j\in\setu}) \leq \max\{0,m - w_{\overline{\setu}}\},
\end{equation*}
and $ \widetilde{t}_{\setu} \le \min\{m,w_{\overline{\setu}} +t_{\setu}\}$.

\medskip

Furthermore, if, for a non-empty $\setu\subseteq [s]$, the projection of $\cP$ onto the components in $\setu$ is a strict digital $(t_{\setu},m,\abs{\setu})$-net, it is true that
\begin{equation*}
    \rho_m ((\CT_j^{(m)})_{j\in\setu}) \leq \max\{0,m - \max\{t_{\setu},w_{\overline{\setu}}\}\}.
\end{equation*}
\end{corollary}

\section{Applications of column reduced digital nets} \label{sec:app_col_red}

\subsection{A reduced matrix product algorithm} \label{subsec:red_mat_column}

In this section, we return to the problem outlined in Section \ref{sec:intro}. Let $P$ be a digital $(t,m,s)$-net over $\FF_b$, with 
generating matrices $C_1^{(m)},\ldots,C_s^{(m)}$. Let $\bsw=(w_j)_{j=1}^s \in \NN_0^s$ be a sequence of reduction indices with $0=w_1\le w_2 \le \dots \le w_s$. 
Let $s^*\le s$ be the largest index such that $w_{s^*} < m$. 
Let $\CT_1^{(m)},\ldots,\CT_s^{(m)}$ be the reduced generating matrices corresponding to $w_1,\ldots,w_s$, and let $Q$ be the corresponding 
reduced digital net. Let $\bsx_0,\ldots,\bsx_{N-1}$ be the points of $Q$, where we interpret $\bsx_0,\ldots,\bsx_{N-1}$ as column vectors. 
Let 
\[
 X=[\bsx_0^\top,\bsx_1^\top,\ldots,\bsx_{N-1}^\top ]^\top
\]
be the $N\times s$ matrix whose $k$-th row is the $k$-th point of $Q$ for $0\le k\le N-1$.

Let $\bsxi_j$ denote the $j$-th column of $X$, i.e., $X=[\bsxi_1,\bsxi_2,\ldots,\bsxi_s]$. Let $A=[\bsa_1,\ldots,\bsa_s ]^\top$, where 
$\bsa_j \in \RR^{1 \times \tau}$ is the $j$-th row of $A$. Then we have 
\begin{equation}\label{eq:MM_product}
 XA=[\bsxi_1,\bsxi_2,\ldots,\bsxi_s] \cdot [\bsa_1,\ldots,\bsa_s ]^\top=
 \bsxi_1 \bsa_1 + \bsxi_2 \bsa_2 + \cdots +\bsxi_s \bsa_s.
\end{equation}
We will make use of a certain inherent repetitiveness of the reduced net $Q$, which we 
will illustrate by considering a reduction index $0 \le w_j < m$ for $1\le j\le s^*$, and the corresponding generator matrix $\CT_j^{(m)}$. The $j$-th components of the $N=b^m$ points of $Q$ (i.e., the $j$-th column $\bsxi_j$ of $X$) are then given by
\begin{eqnarray*}
 \bsxi_j &=& \left(\left(\CT_j^{(m)} \overrightarrow{0}  \right)\cdot(b^{-1},\ldots,b^{-m}),\ldots,
             \left(\CT_j^{(m)} \overrightarrow{(b^m-1)}  \right)\cdot(b^{-1},\ldots,b^{-m})\right)^\top \\
&=& (\underbrace{X_j,\ldots,X_j}_{b^{w_j}\ \rm times})^\top, 
\end{eqnarray*}
where, as above, we write $\overrightarrow{k}$ to denote the vector of base $b$ digits of length $m$ for $k\in\{0,1\ldots,b^m-1\}$, and where
\[
 X_j= \left(\left(\CT_j^{(m)} \overrightarrow{0}  \right)\cdot(b^{-1},\ldots,b^{-m}),\ldots,
             \left(\CT_j^{(m)} \overrightarrow{(b^{m-w_j}-1)}  \right)\cdot(b^{-1},\ldots,b^{-m})\right)^\top.
\]
The reason for this repetitive structure is that, for any $w_j$ with $0 <w_j < m$, the last $w_j$ columns of $\CT_j^{(m)}$ are equal to zero, and thus, in the product $\CT_j^{(m)}\overrightarrow{k}$, the last $w_j$ entries of $\overrightarrow{k}$ become irrelevant. We will exploit this structure within $Q$ to derive a fast matrix-matrix multiplication algorithm to compute $XA$. 

Based on the above observations, it is possible to formulate the following algorithm to compute \eqref{eq:MM_product} in an efficient way.
Note that for $j>s^*$ the $j$-th column of $X$ consists only of zeros, so there is nothing to compute for the entries of $X$ corresponding 
to these columns. 

\begin{algorithm}[H]
\caption{Fast reduced matrix-matrix product using column reduced digital nets}
\label{alg:fast-mv-prod}
\vspace{5pt}

\begin{flushleft}
\textbf{Input:} 

Matrix $A \in \R^{s \times \tau}$, integer $m \in \N$, prime $b$, reduction indices $0=w_1\le w_2 \le \cdots \le w_s$, 
corresponding generating matrices $\CT_1^{(m)},\ldots,\CT_s^{(m)}$ of a reduced digital net.
\end{flushleft}

\begin{algorithmic}
	\STATE Set $N=b^m$ and set $P_{s^*+1} = \bszero_{1 \times \tau} \in \R^{1 \times \tau}$.
	\FOR{$j=s^*$ {\bf to} $1$}
	\STATE $\bullet$ Compute $X_j$ as
	\begin{equation*}
	X_j = \left(\left(\CT_j^{(m)} \overrightarrow{0}  \right)\cdot(b^{-1},\ldots,b^{-m}),\ldots, \left(\CT_j^{(m)} \overrightarrow{(b^{m-w_j}-1)}  \right)\cdot(b^{-1},\ldots,b^{-m})\right)^\top \in \R^{ b^{m-w_j} \times 1} .
	\end{equation*}
	\vspace{-5pt}
	\STATE $\bullet$ Compute $P_j$ as
	\begin{equation*}
	P_j =
	\rotatebox[origin=c]{90}{$b^{\min(w_{j+1},m)-w_j}$ times}\left\{
	\begin{pmatrix}
	P_{j+1} \\
	P_{j+1} \\
	\vdots \\
	P_{j+1} \\
	\end{pmatrix}
	\right.
	+
	X_j  \bsa_j \in \R^{b^{m-w_j} \times \tau},
	\end{equation*}
	\hspace{5pt} where $\bsa_j \in \R^{1 \times \tau}$ denotes the $j$-th row of the matrix $A$.
	\ENDFOR
	\STATE Set $P = P_1$.
\end{algorithmic}
\vspace{5pt}

\begin{flushleft}
 \textbf{Return:} Matrix product $P = X A$.
\end{flushleft}
\end{algorithm}
\begin{remark}\label{remark:theoretical_runtime}
 The number of computations needed for Algorithm \ref{alg:fast-mv-prod} is of order 
 \[
  \calO \left( \sum_{j=1}^{s^*} b^{m-w_j} (\tau + m(m-w_j) ) \right).
 \]
Note that this algorithm also generates the points of the reduced digital net, whereas the standard multiplication or the analogous ``row reduced algorithm'' \cite[Algorithm 4]{DEHKL23}, both require pre-computed points of the digital net as input. Generating the points of 
a non-reduced digital net requires $\calO (b^m s m^2) $ operations, see also \cite[Algorithm 3]{DEHKL23} and the standard 
non-reduced matrix-matrix multiplication usually requires $\calO(b^m s \tau)$ operations. Therefore, \cref{alg:fast-mv-prod} improves the runtime of both steps. We also point out that the number of operations necessary for Algorithm \ref{alg:fast-mv-prod} is independent of $s$, and only depends on $s^*$. If the reduction indices $w_j$ grow sufficiently fast, then $s^*$ can be significantly lower than $s$.   
\end{remark}
\begin{remark}\label{remark:point_transform}
Let us consider mappings $\boldsymbol{\phi}: [0,1]^s \rightarrow\RR$ of the form 
$\boldsymbol{\phi}(\bsx)=(\phi_1 (x_1),\ldots,\phi_s (x_s))$ that we apply simultaneously to all sample points of a given digital net. In this 
case, we can easily adapt Algorithm \ref{alg:fast-mv-prod} such that 
we obtain a reduced net with points transformed by $\boldsymbol{\phi}$, but do not change the order of the computation time outlined in Remark \ref{remark:theoretical_runtime}. Such an adaption is useful when considering the case $D=\RR^s$, as pointed out in the introduction.
\end{remark}

\subsection{Error analysis} \label{subsec:err_ana}

In the beginning of the paper we set out the task 
of approximating the integral \eqref{eq:integral} by the QMC rule \eqref{eq:QMC_approx}. We have shown in the previous sections how to speed up the computation of the products $\bsx_k^\top A$ if we choose $\bsx_k$ as the points of a column reduced digital net. However, we should also keep in mind the integration error made by using a QMC rule of the form \eqref{eq:QMC_approx} using those $\bsx_k$. 

In this section, we restrict ourselves to the case $D=[0,1]^s$, such that we do not need to transform the sample points $\bsx_k$ before applying the corresponding QMC rule. In many applications of quasi-Monte Carlo, one considers so-called \textit{weighted function spaces} such as weighted Sobolev or weighted Korobov spaces (see, e.g., \cite{DKP22, DKS13, DP10}). The idea of studying weighted function spaces goes back to the seminal paper \cite{SW98} of 
Sloan and Wo\'{z}niakowski. The motivation for weighted spaces is that in many applications, different coordinates or different groups of coordinates may have different influence on a multivariate problem. To give a simple example, consider numerical integration of a function $f:[0,1]^s\to\RR$, where 
\[
  f(x_1,\ldots,x_s)=\mathrm{e}^{x_1}+ \frac{x_2 + \cdots + x_s}{2^s}.
\]
Clearly, for large $s$, the first variable has much more influence on this problem than the others. 
In order to make such observations more precise, one introduces weights, which are nonnegative real numbers 
$\gamma_{\fraku}$, one for each set $\fraku \subseteq \{1,\ldots,s\}$. Intuitively speaking, the number $\gamma_{\fraku}$ models the influence 
of the variables with indices in $\fraku$. Larger values of $\gamma_{\fraku}$ mean more influence, smaller values less influence. Formally, we
set $\gamma_{\emptyset}=1$, and we write $\bsgamma=\{\gamma_{\fraku}\}_{\fraku\subseteq \{1,\ldots, s  \}}$. These weights can now be used 
to modify the norm in a given function space, thereby modifying the set over which a suitable error measure, as for example the \textit{worst-case error}, of a problem is considered. By making 
this set smaller according to the weights (in the sense that also here, certain groups of variables may have less influence than others), a problem may thus become easier to handle and even lose the curse of dimensionality, provided that suitable conditions
on the weights hold. This effect also corresponds to intuition---if a problem depends
on many variables, of which only some have significant influence, it is natural to
expect that the problem will be easier to solve than one where all variables have the
same influence.

The \emph{weighted star discrepancy} is (via the well-known \emph{Koksma-Hlawka inequality} or its weighted version, see, e.g., \cite{DKP22,DP10,N92}), a measure of the worst-case quadrature error for a QMC rule with node set $Q$, with $b^m$ nodes, defined as
\begin{equation}
    D_{b^m,\bsgamma}^{*}(Q) := \sup_{x \in (0,1]^s} \max_{\emptyset \neq \fraku \subseteq [s]} \gamma_{\fraku} \abs{\Delta_{Q,\fraku}(\bsx)}, 
\end{equation}
where
\begin{equation}\label{def:DeltaP}
    \Delta_{Q,\fraku}(\bsx) := \frac{ \# \{(y_1,\ldots,y_s)\in Q  \colon y_j < x_j, \, \forall j \in \fraku \}}{b^m}- \prod_{j\in \fraku} x_j.
\end{equation}
Indeed, for certain weighted function classes based on Sobolev spaces of smoothness one, the weighted star discrepancy equals the worst-case quadrature error of a QMC rule with node set $Q$. Here, by the worst-case error, we mean the supremum of the integration error taken over the unit ball of the function class under consideration. We refer to \cite[Section 5.3]{DKP22} for further details on the weighted Koksma-Hlawka inequality.

As shown in \cite{P18}, we have 
\begin{eqnarray*}
    D_{b^m,\bsgamma}^{*}(Q) &=&  \max_{\emptyset \neq \fraku \subseteq [s]} 
    \sup_{\bsx \in (0,1]^s} \gamma_{\fraku} \abs{\Delta_{Q,\fraku}(\bsx)}\\
    &=&  \max_{\emptyset \neq \fraku \subseteq [s]} \gamma_{\fraku}
    \sup_{\bsx \in (0,1]^s}  \abs{\Delta_{Q,\fraku}(\bsx)}.
\end{eqnarray*}
In the latter expression, the suprema over $\bsx \in (0,1]^s$ just yield the values of the star discrepancy of the projections of $Q$, and thus, one can use existing discrepancy bounds for the projections of $Q$. Let us proceed as follows. 
Assume that $\cP$ is a digital $((t_\setu)_{\setu \subseteq [s]},m,s)$-net over $\Fb$ with $m \times m $ generating matrices $C_1^{(m)}, \dots ,C_s^{(m)}$ derived from a digital $((t_\setu)_{\setu \subseteq [s]},s)$-sequence, 
where $m\ge t$. Let $\widetilde{\cP}$ be the corresponding column reduced digital net based on the reduction indices $0=w_1 \le w_2 \le \cdots \le w_s$, 
and let $(\widetilde{t}_\setu)_{\setu \subseteq [s]}$ be the minimal quality parameters of the projections of $\widetilde{\cP}$.

Whenever we consider a $\setu \subseteq [s]$ that is not a subset of $[s^*]$, we know due to Corollary \ref{lem:t_value_red_proj} that the quality parameter of the corresponding 
projection of $\widetilde{P}$ is $m$ and therefore we can bound its discrepancy only trivially by $1$. Whenever we have $\setu \subseteq [s^*]$, however, we can use existing discrepancy bounds for the corresponding net. To this end, we use the results from \cite{FK13}, which are, to our best knowledge, the currently best-known general upper discrepancy bounds for $(t,m,s)$-nets. This yields, for any non-empty set $\setu\subseteq [s]$,
\begin{equation}\label{eq:disc_bound_proj}
 \sup_{\bsx \in (0,1]^s}  \abs{\Delta_{\widetilde{\cP},\fraku}(\bsx)} \le 
 \begin{cases}
  1 & \mbox{if $\setu\not\subseteq [s^*]$,}\\
  ( b^{\widetilde{t}_\setu} / b^m )\sum_{v=0}^{\abs{\setu}-1}   a_{v,b}^{(\abs{\setu})} m^v & \mbox{if $\setu\subseteq [s^*]$ and $\abs{\setu}\ge 2$,}\\
   b^{\widetilde{t}_\setu} / b^m  & \mbox{if $\setu\subseteq [s^*]$ and $\abs{\setu}= 1$,}
 \end{cases}
\end{equation}
where
\begin{eqnarray*}
a_{v,b}^{(\abs{\setu})}&=&{\abs{\setu}-2 \choose v} \left(\frac{b+2}{2}\right)^{\abs{\setu}-2-v}\frac{(b-1)^{v}}{2^{v}v!} (a_{0,b}^{(2)}+ \abs{\setu}^2-4)\nonumber\\
&&+{\abs{\setu}-2 \choose v-1} \left(\frac{b+2}{2}\right)^{\abs{\setu}-1-v}\frac{(b-1)^{v-1}}{2^{v-1}v!} a_{1,b}^{(2)},
\end{eqnarray*}
for $0\le v\le \abs{\setu}-1$,
with
\[
a_{0,b}^{(2)}=\begin{cases}
                                     \frac{b+8}{4}&\mbox{if $b$ is even,}\\ \\
                                     \frac{b+4}{2}&\mbox{if $b$ is odd,}
                                     \end{cases}
\hspace{1cm} \mbox{and} \hspace{1cm} 
a_{1,b}^{(2)}=\begin{cases}
                                     \frac{b^2}{4(b+1)}&\mbox{if $b$ is even,}\\ \\
                                     \frac{b-1}{4}&\mbox{if $b$ is odd.}
                                     \end{cases}
\]
This then yields
\begin{equation}\label{eq:disc_bound_global}
 D_{b^m,\bsgamma}^{*}(\widetilde{\cP}) \le  \max\left\{
 \max_{\substack{\emptyset \neq \fraku \subseteq [s]\\ \setu\not\subseteq [s^*]}} \gamma_{\fraku}
 ,\max_{\substack{\fraku \subseteq [s^*]\\ \abs{\setu}=1}} \gamma_{\setu}\,
 \frac{b^{\widetilde{t}_\setu}}{b^m}
 ,\max_{\substack{\fraku \subseteq [s^*]\\ \abs{\setu}\ge 2}} \gamma_{\setu}\,
 \frac{b^{\widetilde{t}_\setu}}{b^m}\sum_{v=0}^{\abs{\setu}-1}   a_{v,b}^{(\abs{\setu})} m^v 
 \right\}.
\end{equation}

Let us analyze the three maxima in the curly brackets in \eqref{eq:disc_bound_global} in greater detail. To this end, as also in \cite{DEHKL23},
we restrict ourselves to product weights in the following, i.e., we assume weights $\gamma_{\fraku} = \prod_{j\in \fraku} \gamma_j$ with $\gamma_1 \ge \gamma_2\ge  \cdots >0$.

For the first term, we proceed as in \cite{DEHKL23}, namely we use that $w_j\ge m$ if $j\in \setu\setminus [s^\ast]$, and obtain for $\setv = \setu \cap [s^\ast]$ that
\begin{equation}\label{eq:disc_term_1}
 \gamma_{\setu} \le \gamma_{\setv} \gamma_{\setu\setminus\setv} \frac{1}{b^m}\prod_{j\in\setu\setminus\setv} (1+b^{w_j})
 \le \frac{1}{b^m}  \prod_{j\in \setu} \gamma_j (1 + b^{w_j}). 
\end{equation}

For the second maximum in \eqref{eq:disc_bound_global}, note that we have 
$\setu=\{j\}$ for some $j\in [s^*]$, and hence
$\widetilde{t}_{\setu} \le \min\{m,w_j + t_{\{j\}}\}$ by Corollary \ref{lem:t_value_red_proj}. Consequently,
\begin{equation}\label{eq:disc_term_2}
 \max_{\substack{\fraku \subseteq [s^*]\\ \abs{\setu}=1}} \gamma_{\setu}
 \frac{b^{\widetilde{t}_\setu}}{b^m} \le 
 \max_{j\in [s^*]} \gamma_j \frac{b^{\min\{m,w_j + t_{\{j\}}\}}}{b^m}.
\end{equation}

For the third maximum in \eqref{eq:disc_bound_global}, we again use 
Corollary \ref{lem:t_value_red_proj}, and obtain 
\begin{equation}\label{eq:disc_term_3}
 \max_{\substack{\fraku \subseteq [s^*]\\ \abs{\setu}\ge 2}} \gamma_{\setu}\,
 \frac{b^{\widetilde{t}_\setu}}{b^m}\sum_{v=0}^{\abs{\setu}-1}   a_{v,b}^{(\abs{\setu})} m^v
 \le 
 \max_{\substack{\fraku \subseteq [s^*]\\ \abs{\setu}\ge 2}} \gamma_{\setu}\,
 \frac{b^{\min \{m, w_{\overline{\setu}}+t_{\setu}\}}}{b^m}\sum_{v=0}^{\abs{\setu}-1}   a_{v,b}^{(\abs{\setu})} m^v.
\end{equation}
Using these estimates in \eqref{eq:disc_bound_global}, we obtain
\begin{eqnarray}\label{eq:disc_bound_global_2}
 & D_{b^m,\bsgamma}^{*}(\widetilde{\cP})& \nonumber \\ &\le& \!\!\!\!\!\!\! \max\left\{
 \max_{\substack{\emptyset \neq \fraku \subseteq [s]\\ \setu\not\subseteq [s^*]}}
  \frac{1}{b^m}  \prod_{j\in \setu} \gamma_j (1 + b^{w_j}),
 \max_{j\in [s^*]} \gamma_j \frac{b^{w_j + t_{\{j\}}}}{b^m}
 ,\max_{\substack{\fraku \subseteq [s^*]\\ \abs{\setu}\ge 2}} \gamma_{\setu}\,
 \frac{b^{\min \{m, w_{\overline{\setu}}+t_{\setu}\}}}{b^m}\sum_{v=0}^{\abs{\setu}-1}   a_{v,b}^{(\abs{\setu})} m^v
 \right\}.\nonumber\\
\end{eqnarray}

\begin{remark}
 A few remarks on \eqref{eq:disc_bound_global_2} are in order. Note that only 
 the first term in the curly brackets in \eqref{eq:disc_bound_global_2} depends on $s$. The two remaining terms depend on $s^*$, which can be independent of $s$ if the reduction indices $w_j$ increase sufficiently fast. However, let us give a few further details on these observations.

We may want that the first term
 \[
  \frac{1}{b^m}  \prod_{j\in \setu} \gamma_j (1 + b^{w_j}) \le \frac{1}{b^m}  \prod_{j=1}^s \gamma_j (1 + b^{w_j})
 \]
be bounded by $\kappa/b^m$ for some constant $\kappa> 0$ independent of $s$. Let $j_0\in \NN$ be minimal such that $\gamma_j \le 1$ for all $j > j_0$. Then we impose $ \prod_{j=1}^{s} \gamma_j(1 + b^{w_j})\le  \gamma_1^{j_0}\prod_{j=1}^{s} (1 + \gamma_j b^{w_j})\le \kappa $. Hence it is sufficient to choose $\kappa >  \gamma_1^{j_0}$ and  for all $j \in [s]$, 
\begin{equation}\label{eq:wchoice}
w_j := \min\left(\floor{\log_b\left(\frac{\left(\frac{\kappa}{ \gamma_1^{j_0}}\right)^{1/s}-1}{\gamma_j}\right)},m\right).
\end{equation}
The choice of the $w_j$ in \eqref{eq:wchoice} depends on $s$. For sufficiently fast decaying weights $\gamma_j$, it is possible to choose the $w_j$ such that they no longer depend on $s$. Indeed, suppose, e.g., that $\gamma_j=j^{-2}$. Then we could choose the $w_j$ such that, for some $\tau\in(1,2)$,
\begin{equation}\label{eq:wchoice_2}
  w_j \le \min\left(\floor{\log_b\left( j^{2-\tau}\right)},m\right).
\end{equation}
This then yields
\begin{eqnarray*}
  \prod_{j=1}^s (1+\gamma_j b^{w_j})\le \exp\left(\sum_{j=1}^s \log (1+\gamma_j b^{w_j})\right)
  \le \exp\left(\sum_{j=1}^s \gamma_j b^{w_j}\right)\le \exp (\zeta(\tau)),
\end{eqnarray*}
where $\zeta(\cdot)$ is the Riemann zeta function. This gives a dimension-independent bound on the term $\prod_{j=1}^s \gamma_j (1+b^{w_j})$ from above, and hence a dimension-independent bound for all of $D_{b^m,\bsgamma}^{*}(\widetilde{\cP})$.

Regarding the second term in \eqref{eq:disc_bound_global_2}, this term only depends on one-dimensional projections of $\widetilde{\cP}$. In particular, if we choose the $w_j$ as in \eqref{eq:wchoice_2}, this expression should be easy to bound from above. This is even more so if the 
$t$-values of the one-dimensional projections of the non-reduced net $\cP$ are low, which may often be the case (in fact, the $t$-values of one-dimensional projections might even be zero in many examples). Thus we can bound the second term by an expression 
of the form $\kappa^* / b^m$, which only depends on $s^*$ but not on $s$.

Regarding the third term in \eqref{eq:disc_bound_global_2}, it crucially depends on the weights $\bsgamma$ and their interplay with the quality parameters of the projections of $\cP$, $t_{\setu}$. In particular, small quality parameters, in combination with sufficiently fast decaying weights and a suitable choice of the reduction indices $w_j$,
should yield tighter error bounds. Indeed, we could proceed similarly to \cite[Corollary 1]{FK13}, and bound the third term in \eqref{eq:disc_bound_global_2} by a term of the form 
\[
\max_{\substack{\fraku \subseteq [s^*]\\ \abs{\setu}\ge 2}} \gamma_{\setu}\,
\frac{1}{b^m} \left(  c_{\abs{\setu}} m^{\abs{\setu} -1} + \calO (m^{\abs{\setu} -2}) \right),
\]
where $c_{\abs{\setu}}$ depends on $b,{t_{\setu}}$ and $\abs{\setu}$, but not on $m$. Note that also the third term only depends on $s^*$ and not on $s$, so for sufficiently fast increasing reduction indices $w_j$, the dimension $s$ does not matter.
In summary, we obtain 
\[  
    D_{b^m,\bsgamma}^{*}(\widetilde{\cP}) \le \max\left\{ \frac{\kappa}{b^m}, 
    \frac{\kappa^*}{b^m}, \max_{\substack{\fraku \subseteq [s^*]\\ \abs{\setu}\ge 2}} \gamma_{\setu}\,
\frac{1}{b^m} \left(  c_{\abs{\setu}} m^{\abs{\setu} -1} + \calO (m^{\abs{\setu} -2}) \right)
    \right\}.
\]

\end{remark}
 
 \begin{remark}\label{rem_th_advantage}
 Note that our new result yields an advantage over the corresponding result for row reduced nets in \cite{DEHKL23}. In that paper, one needs to work with the quality parameters of the projections of the reduced net, which are, in general, not known. In the present paper, we benefit from the combination of the column reduction and the fact that the nets considered here are derived from digital sequences, which guarantees additional structure. Usually, it is computationally involved to determine the $t$-value of a digital net or sequence from the generating matrices, since many linear independence conditions need to be checked.
 Here, however, we can use Theorem \ref{lem:t_value_red} and Corollary \ref{lem:t_value_red_proj}, which relate the $t$-values of $\cP$ to those of $\widetilde{\cP}$, and thus give us an advantage. In particular, if $\cP$ is obtained from, say, a Sobol' or a Niederreiter sequence, it should be possible to have $t$-values that are guaranteed to be reasonably low.
\end{remark}
 
\section{Numerical experiments}
In this section, we test the computational performance of column reduced digital nets for matrix products $XA$, where $A$ is an $s\times \tau$ matrix, as detailed in \cref{subsec:red_mat_column}.
We implemented \cref{alg:fast-mv-prod} in the Julia programming language (Version 1.9.3).\footnote{Source code available at \url{https://github.com/Vishnupriya-Anupindi/ReducedDigitalNets.jl}}
In the following plots, we compare the runtime of \cref{alg:fast-mv-prod} to the standard matrix multiplication and also the matrix multiplication using the points from row reduced digital nets as given in \cite[Algorithm 4]{DEHKL23}. 
We remark that the reported runtimes are also affected by technical implementation details such as memory efficiency, a detailed discussion of which is out of scope here.

For the generating matrices $C_1^{(m)},\ldots,C_s^{(m)}$, we used random matrices in $\Fb^{m \times m}$, since the matrix product computation itself does not depend on the entries of the matrix, i.e, we get similar relations of runtimes if we use generating matrices of specific digital sequences like Sobol' or Niederreiter sequences. 

In \cref{fig:weight_comp} we see, for fixed $b = 2, m = 12$, and $\tau = 20$, how the runtime changes as we vary $s$. We compare this for two different choices of reduction indices $w_j$. We see that in this case, using column reduced digital nets in \cref{alg:fast-mv-prod} performs better than the use of row reduced digital nets in \cite[Algorithm 4]{DEHKL23} and also the standard matrix multiplication. 

As the reduction indices $w_j$ increase more slowly (as in \cref{fig:wt_inc_2}), the difference in performance between the standard multiplication and  \cref{alg:fast-mv-prod} reduces. We can see this also theoretically by inserting the weights in \cref{remark:theoretical_runtime}.

\begin{figure}[ht]
	\begin{subfigure}[t]{0.48\textwidth}
		\includegraphics[width=\textwidth]{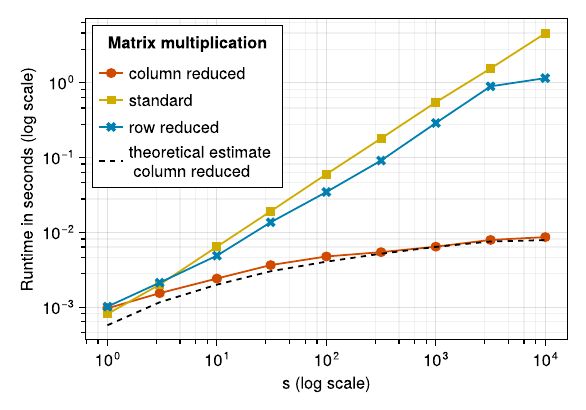}
		\subcaption{$w_j = \min(\lfloor \log_2(j) \rfloor, m)$}
	\end{subfigure}
	\hfill
	\begin{subfigure}[t]{0.48\textwidth}
		\includegraphics[width=\textwidth]{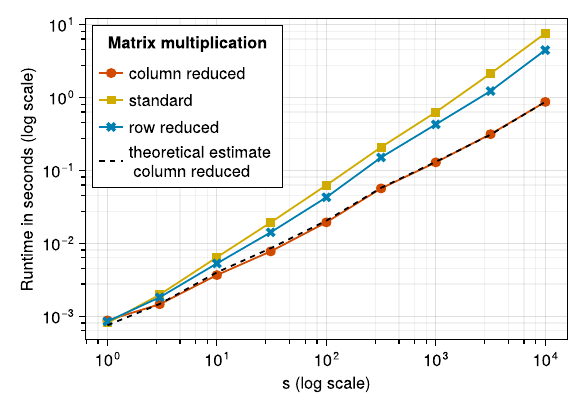}
		\subcaption{$ w_j = \min(\lfloor \log_2(j^{1/2}) \rfloor, m)$}
        \label{fig:wt_inc_2}
	\end{subfigure}
\caption{$m=12,\tau = 20$, varying $w_j$} \label{fig:weight_comp}
\end{figure}

\begin{figure}[h]
    \centering
		\includegraphics[width=0.48\textwidth]{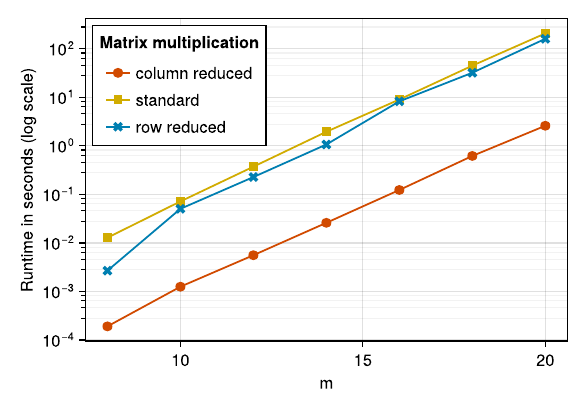}
	\caption{$s=800, \tau = 20$, varying $m$} \label{fig:vary_m}
\end{figure}

In \cref{fig:vary_m}, we study the behavior for fixed $b = 2, s = 800$, and $\tau = 20$ as $m$ increases. Note that we use the logarithmic scale for the time but not for $m$. We observe that also in this case \cref{alg:fast-mv-prod} seems to perform better than the row reduced case.

Overall, the numerical tests for the runtime using column reduced digital nets fit our theoretical estimate for the runtime as given in \cref{remark:theoretical_runtime} and comparison with the row reduced algorithm reveals that the column reduced algorithm could yield a better performance. Additionally to this practical advantage, column reduced matrices also have a theoretical advantage over row reduced matrices, as pointed out in Remark \ref{rem_th_advantage}.

\section{Conclusion}
 Column reduced digital nets have applications in the field of quasi-Monte Carlo methods. We can speed up the matrix-matrix multiplication in the quasi-Monte Carlo method by exploiting the repetitive structure of the points of a column reduced digital net. The bounds for the quality parameter ($t$-value) of column reduced digital nets have not been studied before.

In our research, we provide an algorithm for the matrix-matrix product using column reduced digital nets, which is faster than the standard matrix multiplication algorithm. In addition, we provide bounds for the $t$-value for column reduced digital nets. This is very essential for the error analysis of our method and has an advantage over the corresponding result for the row reduced nets in \cite{DEHKL23}. 

For future work, one could consider relaxing the conditions we impose on the $t$-value of the underlying digital sequence. One could also explore in-depth the interplay between column and row reduced digital nets. 

\section*{Acknowledgments}

The authors acknowledge the support of the Austrian Science Fund (FWF) Project P34808. For open access purposes, the authors have applied a CC BY public copyright license to any author accepted manuscript version arising from this submission.

\begin{small}
	\noindent\textbf{Authors' address:}\\
	
	\noindent Vishnupriya Anupindi\\ 
        Peter Kritzer\\
	Johann Radon Institute for Computational and Applied Mathematics (RICAM)\\
	Austrian Academy of Sciences\\
	Altenbergerstr. 69, 4040 Linz, Austria.\\
        \texttt{vishnupriya.anupindi@ricam.oeaw.ac.at}\\
	\texttt{peter.kritzer@ricam.oeaw.ac.at}\\

\end{small}

\end{document}